\documentclass[a4paper]{amsart} 
\usepackage{amssymb} 
\usepackage[all]{xy}
\SelectTips{cm}{}

\usepackage{hyperref}
\calclayout

\title{Morphisms determined by objects in triangulated categories}

\author{Henning Krause}
\address{Henning Krause\\ Fakult\"at f\"ur Mathematik\\
Universit\"at Bielefeld\\ D-33501 Bielefeld\\ Germany.}
\email{hkrause@math.uni-bielefeld.de}

\newtheorem{lem}{Lemma}[section]
\newtheorem{prop}[lem]{Proposition}
\newtheorem{cor}[lem]{Corollary}
\newtheorem{thm}[lem]{Theorem}

\theoremstyle{remark}
\newtheorem{rem}[lem]{Remark}

\theoremstyle{definition}
\newtheorem{exm}[lem]{Example}
\newtheorem{defn}[lem]{Definition}
\newtheorem{question}[lem]{Question}

\numberwithin{equation}{section}

\renewcommand{\mod}{\operatorname{\mathsf{mod}}\nolimits}

\newcommand{\proj}{\operatorname{\mathsf{proj}}\nolimits}

\newcommand{\rad}{\operatorname{rad}\nolimits}

\newcommand{\add}{\operatorname{\mathsf{add}}\nolimits}

\newcommand{\Mod}{\operatorname{\mathsf{Mod}}\nolimits}

\newcommand{\End}{\operatorname{End}\nolimits}
\newcommand{\Hom}{\operatorname{Hom}\nolimits}

\renewcommand{\Im}{\operatorname{Im}\nolimits}
\newcommand{\Ker}{\operatorname{Ker}\nolimits}
\newcommand{\Coker}{\operatorname{Coker}\nolimits}

\newcommand{\res}{\operatorname{res}\nolimits}
\newcommand{\Ext}{\operatorname{Ext}\nolimits}

\newcommand{\Tr}{\operatorname{Tr}\nolimits}

\newcommand{\uHom}{\operatorname{\underline{Hom}}\nolimits}

\newcommand{\Ab}{\mathsf{Ab}}
\newcommand{\op}{\mathrm{op}}

\newcommand{\coind}{\operatorname{coind}\nolimits}

\newcommand{\lto}{\longrightarrow}
\newcommand{\xto}{\xrightarrow}

\def\a{\alpha}
\def\b{\beta}

\def\p{\phi}

\def\Ga{\Gamma}
\def\La{\Lambda}

\def\C{{\mathsf C}}
\def\D{{\mathsf D}}

\def\Sc{{\mathsf S}}

\def\bbZ{{\mathbb Z}}

\begin{document}

\begin{abstract}
  The concept of a \emph{morphism determined by an object} provides a
  method to construct or classify morphisms in a fixed category.  We
  show that this works particularly well for triangulated categories
  having Serre duality. Another application of this concept arises
  from a reformulation of Freyd's generating hypothesis.
\end{abstract}

\maketitle
\setcounter{tocdepth}{1}
\tableofcontents

\section{Introduction}

Given a category, one may ask for a classification of all morphisms
ending in a fixed object, where two such morphisms $\a_i\colon X_i\to
Y$ $(i=1,2)$ are \emph{isomorphic} if there exists an isomorphism
$\p\colon X_1\to X_2$ such that $\a_1=\a_2\p$.  In this note we dicuss
an approach which is based on the concept of a \emph{morphism
  determined by an object}.  Originally, this concept was introduced
by Auslander \cite{Au1978} in order to give a conceptual explanation
for the existence of left and right almost split morphisms introduced
before in joint work with Reiten \cite{AuRe1975}.  Here, we show that
some of Auslander's results have an analogue for triangulated
categories. In a somewhat different direction, we reformulate a
conjecture of Freyd from stable homotopy theory in terms of morphisms
determined by objects.

\section{Morphisms determined by objects: Auslander's work}

We give a quick review of Auslander's work on morphisms determined by
objects. This leads then to a precise formulation of the classification
problem for morphisms in terms of morphisms determined by objects.

\begin{defn}[Auslander \cite{Au1978}] 
  A morphism $\a\colon X\to Y$ in some fixed category is said to be
  \emph{right determined} by an object $C$ if for every morphism
  $\a'\colon X'\to Y$ the following conditions are equivalent:
\begin{enumerate}
\item The morphism $\a'$ factors through $\a$.
\item For every morphism $\p\colon C\to X'$ the composite $\a'\p$
  factors through $\a$.
\end{enumerate}
A morphism is \emph{left determined} by $C$ if it is right
determined by $C$ when viewed as morphism in the opposite category.
\end{defn}


Fix a morphism $\a\colon X\to Y$. We denote by $\Im\Hom(C,\a)$ the
image of the induced morphism $\Hom(C,X)\to\Hom(C,Y)$ and observe that
condition (2) means \[\Im\Hom(C,\a')\subseteq\Im\Hom(C,\a).\] The
morphism $\a$ is called \emph{right minimal} if every morphism
$\p\colon X\to X$ satisfying $\a=\a\p$ is an isomorphism. 

The following elementary observation yields a reformulation of the
classification problem for morphisms ending in a fixed object.

\begin{lem}\label{le:uni}
  Let $\a_i\colon X_i\to Y$ $(i=1,2)$ be morphisms that are right
  minimal and right $C$-determined. Then $\a_1$ and $\a_2$ are
  isomorphic if and only if $\Im\Hom(C,\a_1)=\Im\Hom(C,\a_2)$.\qed
\end{lem}

Right almost split morphisms provide an important class of examples of
right determined morphisms.

\begin{exm}[Auslander {\cite[\S II.2]{Au1978}}]\label{ex:ass}
  A morphism $\a\colon X\to Y$ in some additive category is
  \emph{right almost split} (that is, $\a$ is not a retraction and
  every morphism $X'\to Y$ that is not a retraction factors through
  $\a$) if and only if $\Ga=\End(Y)$ is a local ring, $\a$ is right
  determined by $Y$, and $\Im\Hom(Y,\a)=\rad\Ga$.
\end{exm}

Let us recall two of the main results from \cite{Au1978,Au1978b}. Fix
a ring $\La$ and denote by $\Mod\La$ the category of
$\La$-modules. The full subcategory formed by all finitely presented
$\La$-modules is denoted by $\mod\La$.

\begin{thm}[Auslander  {\cite[Theorem~I.3.19]{Au1978}}]\label{th:au}
  Let $C$ and $Y$ be $\La$-modules and suppose that $C$ is finitely
  presented. Given an $\End_\La(C)$-submodule
  $H\subseteq\Hom_\La(C,Y)$, there exists, up to isomorphism, a unique
  right minimal morphism $\a\colon X\to Y$ in $\Mod\La$ which is right
  $C$-determined and satisfies $\Im\Hom_\La(C,\a)=H$.\qed
\end{thm}

\begin{thm}[Auslander {\cite[Theorem~2.6]{Au1978b}}]
  Suppose that $\La$ is an Artin algebra and let $\a\colon X\to Y$ be
  a morphism between finitely presented $\La$-modules. Denote by $\Tr
  D(\Ker\a)$ the transpose of the dual of the kernel of $\a$ and by
  $P(\Coker\a)$ a projective cover of the cokernel of $\a$.  Then $\a$
  is right determined as a morphism in $\mod\La$ by $\Tr
  D(\Ker\a)\amalg P(\Coker\a)$.\footnote{This result is not correct as
    stated; the term $P(\Coker\a)$ needs to be modified, as pointed
    out by Ringel in \cite{Ri2011}.}\qed
\end{thm}

Note that the proofs of both theorems are based on the
\emph{Auslander--Reiten formula}
\[D\uHom_\La(C,-)\cong\Ext^1_\La(-,D\Tr C),\] where $C$ is supposed to
be a finitely presented $\La$-module and $D=\Hom_\Ga(-,I)$ with
$\Ga=\End_\La(C)$ and $I$ an injective $\Ga$-module; see
\cite[Proposition~I.3.4]{Au1978}.  For instance, a $C$-determined
epimorphism corresponding to a $\Ga$-submodule
$H\subseteq\Hom_\La(C,Y)$ is obtained by choosing an injective
envelope $\uHom_\La(C,Y)/H\to I$ over $\Ga$ and taking the morphism
ending at $Y$ from the corresponding extension $0\to D\Tr C\to X\to
Y\to 0$. A modification of this construction takes care of morphisms
that are not epimorphisms.

In the following section, we move from module categories to
dualising varieties and obtain similar results using an analogue
of the Auslander--Reiten formula.

Motivated by Auslander's results, the classification problem for
morphisms ending in a fixed object may be formulated more precisely 
in terms of the following definition.

\begin{defn}\label{de:det}
An additive category $\C$ is said \emph{to have right determined
    morphisms} if for every  object $Y\in\C$ the following holds:
\begin{enumerate}
\item Given an object $C\in\C$ and an
  $\End_\C(C)$-submodule  $H\subseteq\Hom_\C(C,Y)$, there
is  a right $C$-determined morphism $\a\colon X\to Y$ with
  $\Im\Hom_\C(C,\a)=H$.
\item Every morphism ending in $Y$ is right
  determined by an object in $\C$.
\end{enumerate}
\end{defn}

In categorical terms, this definition formulates properties of the
\emph{slice category} $\C/Y$ over a fixed object $Y\in\C$.  Let us
illustrate this by looking at categories having few morphisms.

\begin{exm}
  Let $\C$ be a partially ordered set, viewed as a category, and fix a
  morphism $\a\colon x\to y$, which means that $x\le y$. If $x=y$, then $\a$ is
  right determined by every object of $\C$. If $x\neq y$, then $\a$ is
  right determined by an object $c\in\C$ if and only if 
 there exists a unique minimal element in 
\[\C_\a=\{c\in\C\mid c\not\le x,c\le y\}.\]
In that case $c=\inf\C_\a$. Thus in $(\mathbb Z,\le)$ all morphisms
are determined by objects, while in $(\mathbb Q,\le)$ only identity
morphisms are determined by some object.
\end{exm}

\section{Dualising varieties}

Dualising varieties were introduced by Auslander and Reiten in
\cite{AuRe1974}. They form a convenient setting for studying morphisms
determined by objects. We will see that such categories have right
determined morphisms in the sense of Definition~\ref{de:det}.

Throughout this work $k$ denotes a commutative artinian ring with
radical $\mathfrak r$. We fix a $k$-linear additive category $\C$
which is \emph{Hom-finite}, that is, the $k$-module $\Hom_\C(X,Y)$ has
finite length for all objects $X,Y$ in $\C$. Suppose also that $\C$ is
essentially small and idempotent complete.

\subsection*{Dualising varieties}

Let $\mod k$ denote the category of finitely presented $k$-modules and
fix an injective envelope $E=E(k/\mathfrak r)$ over $k$. This provides
the duality
\[D=\Hom_k(-,E)\colon\mod k\lto\mod k.\]

Denote by $(\C,\mod k)$ the category of $k$-linear functors $\C\to\mod
k$.  The basic tools are the fully faithful \emph{Yoneda functor}
\[\C\lto (\C^\op,\mod k),\quad X\mapsto\Hom_\C(-,X),\]
and the duality
\[(\C,\mod k)^\op\stackrel{\sim}\lto (\C^\op,\mod k),\quad
F\mapsto DF.\]

Recall that an additive functor $F\colon\C^\op\to\mod k$ is \emph{finitely presented}
if it fits into an exact sequence
\[\Hom_\C(-,X)\lto\Hom_\C(-,Y)\lto F\lto 0.\]
We denote by $\mod\C$ the full subcategory of $(\C^\op,\mod k)$ formed
by all finitely presented functors.

\begin{defn}[Auslander--Reiten \cite{AuRe1974}]
  A $k$-linear additive Hom-finite essentially small and idempotent
  complete category $\C$ is called \emph{dualising $k$-variety} if the
  assignment $F\mapsto DF$ induces an equivalence
\[(\mod\C)^\op\stackrel{\sim}\lto\mod(\C^\op).\]
\end{defn}

A morphism $X\to Y$ is a \emph{weak kernel} of a morphism $Y\to Z$ in
$\C$ if it induces an exact sequence
\[\Hom_\C(-,X)\lto\Hom_\C(-,Y)\lto\Hom_\C(-,Z).\]
A \emph{weak cokernel} is defined analogously.  The following lemma is
well-known and easily proved.
\begin{lem}\label{le:modC}
  The category $\mod\C$ is an additive category with cokernels; it is
  abelian if and only if $\C$ has weak kernels.\qed
\end{lem}

This yields the following reformulation of the definition of a
dualising variety.

\begin{lem}\label{le:dual}
  Let $\C$ be a $k$-linear additive Hom-finite essentially small and
  idempotent complete category. Then $\C$ is a dualising $k$-variety
  if and only if the following holds:
\begin{enumerate}
\item The category $\C$ has weak kernels and weak cokernels.
\item The functors $D\Hom_\C(-,C)$ and $D\Hom_\C(C,-)$ are finitely
  presented for all objects $C$ in $\C$.\qed
\end{enumerate}
\end{lem} 

\begin{exm}[{Auslander--Reiten \cite[\S2]{AuRe1974}}]
Let $\La$ be an Artin $k$-algebra. Then the category $\proj\La$ of finitely
generated projective $\La$-modules is a dualising $k$-variety.

If $\C$ is a dualising $k$-variety, then $\mod\C$ is a dualising
$k$-variety. In particular, $\mod\La=\mod(\proj\La)$ is a dualising $k$-variety.
\end{exm}

\subsection*{Restriction}
For an object $C$ in $\C$ and $\Ga=\End_\C(C)$, consider the
restriction functor
\[(\C^\op,\mod k)\lto\mod\Ga,\quad F\mapsto F(C)\]
and its right adjoint
\[\coind_C\colon\mod\Ga\lto (\C^\op,\mod k), \quad I\mapsto
\Hom_\Ga(\Hom_\C(C,-),I).\]
Note that Yoneda's lemma gives for each $Y\in\C$ the adjointness isomorphism
\begin{equation}\label{eq:Yon}
\Hom_\C(\Hom_\C(-,Y),\coind_C I)\xto{\sim}\Hom_\Ga(\Hom_\C(C,Y),I),
\quad \eta\mapsto \eta_*
\end{equation}
with
\begin{equation}
\eta_X(\a)=\eta_*\Hom_\C(C,\a)\quad\text{for all}\quad
X\in\C,\,\a\in\Hom_\C(X,Y).
\end{equation}
In particular,
\begin{equation}\label{eq:eta}
\eta\Hom_\C(-,\a)=0\quad\iff\quad\eta_*\Hom_\C(C,\a)=0.
\end{equation}

\begin{lem}\label{le:Yon}
Let $I=D\Ga$. Then
$\coind_C I\cong D\Hom_\C(C,-)$.
\end{lem}
\begin{proof}
One computes
\begin{align*}
\Hom_\Ga(\Hom_\C(C,-),\Hom_k(\Ga,k))&\cong
\Hom_k(\Hom_\C(C,-)\otimes_\Ga\Ga,k)\\
&\cong D\Hom_\C(C,-).\qedhere
\end{align*}
\end{proof}

\subsection*{Finding a determinator of a  morphism}
Following Ringel \cite{Ri2011}, an object that determines a morphism
is called a \emph{determinator}. Our first aim is to find for each
morphism in $\C$ a determinator.

\begin{lem}\label{le:det}
  Fix an object $C\in\C$ and set $\Ga=\End_\C(C)$. Let $\a\colon X\to
  Y$ be a morphism in $\C$ and suppose there is an exact sequence
\[\Hom_\C(-,X)\xto{(-,\a)} \Hom_\C(-,Y)\stackrel{\eta}\lto \coind_C I\] 
for some $I\in\mod\Ga$. Then $\a$ is right $C$-determined.
\end{lem}
\begin{proof}
Fix a morphism $\a'\colon X'\to Y$ such that for every morphism
$\p\colon C\to X'$ the composite $\a'\p$ factors through $\a$.  This means
\[\Im\Hom_\C(C,\a')\subseteq \Im\Hom_\C(C,\a).\]
It follows from \eqref{eq:eta} that $\eta\Hom_\C(-,\a')=0$. Thus $\a'$ factors through
$\a$.
\end{proof}

\begin{prop}\label{pr:det-obj}
Let $\a\colon X\to Y$ be a morphism in $\C$ and  suppose there is an
 exact sequence
\[\Hom_\C(-,X)\xto{(-,\a)} \Hom_\C(-,Y)\lto D\Hom_\C(C,-)\] for some object
$C\in\C$. Then $\a$ is right $C$-determined.
\end{prop}
\begin{proof}
  Observe that $ D\Hom_\C(C,-)=\coind_C I$ for $I=D\Ga$ and
  $\Ga=\End_\C(C)$, by Lemma~\ref{le:Yon}. Now apply
  Lemma~\ref{le:det} to see that $\a$ is right determined by $C$.
\end{proof}

\begin{cor}\label{co:det}
Suppose that $\C$ has weak cokernels and $D\Hom_\C(-,C)$ is finitely
presented for each $C\in\C$. Then every morphism in $\C$ is right
determined by an object in $\C$.
\end{cor}
\begin{proof}
  Fix a morphism $\a$ in $\C$.  The assumptions on $\C$ ensure that
  the functor $D\Coker\Hom_\C(-,\a)$ is finitely presented; see
  Lemma~\ref{le:modC}. This means that there is a monomorphism
  $\Coker\Hom_\C(-,\a)\to D\Hom_\C(C,-)$ for some $C\in\C$.  Thus
$\a$ is right $C$-determined by  Proposition~\ref{pr:det-obj}.
\end{proof}

\subsection*{Finding morphisms  determined by an object}

We construct  morphisms that are determined by a fixed object.

\begin{prop}\label{pr:det-exist}
  Suppose that $\C$ has weak kernels. Fix two objects $C,Y$ in $\C$
  and an $\End_\C(C)$-submodule $H\subseteq\Hom_\C(C,Y)$. Suppose also
  that the functor $D\Hom_\C(C,-)$ is finitely presented. Then there exists a right
  $C$-determined morphism $\a\colon X\to Y$ satisfying
  $\Im\Hom_\C(C,\a)=H$.
\end{prop}

\begin{rem}\label{re:uni}
  The morphism $X\to Y$ in Proposition~\ref{pr:det-exist} can be
  chosen to be right minimal. This follows from the subsequent
  remark. With this choice, the morphism is unique up to isomorphism,
  by Lemma~\ref{le:uni}.
\end{rem}

\begin{rem}\label{re:min}
Given a morphism $\a\colon X\to Y$ in $\C$, there is a decomposition
$X=X'\amalg X''$ such that $\a|_{X'}$ is right minimal and
$\a|_{X''}=0$. This follows from the fact that the endomorphism ring
of every object in $\C$ is semiperfect.
\end{rem}

\begin{proof}[Proof of Proposition~\ref{pr:det-exist}]
  Choose an injective envelope $\Hom_\C(C,Y)/H\to I$ over
  $\Ga=\End_\C(C)$. The composite
\[\Hom_\C(C,Y)\twoheadrightarrow\Hom_\C(C,Y)/H\to I\]
corresponds under the isomorphism \eqref{eq:Yon} to a morphism
  \[\eta\colon \Hom_\C(-,Y)\to \coind_C I.\]
 
  Next observe that $\coind_C I$ is finitely presented since
  $\coind_C(D\Ga)$ is finitely presented, by the assumption on $C$ and
  Lemma~\ref{le:Yon}.  It follows that the kernel of $\eta$ is
  finitely presented since $\mod\C$ is abelian. Thus there is a
  morphism $\a\colon X\to Y$ which yields an exact sequence
 \begin{equation*}
\Hom_\C(-,X)\xto{(-,\a)} \Hom_\C(-,Y)\stackrel{\eta}\lto \coind_C I.
\end{equation*}
Evaluating this sequence at $C$ shows that $\Im\Hom_\C(C,\a)=H$, and
Lemma~\ref{le:det} shows that $\a$ is determined by $C$.
\end{proof}

\begin{cor}\label{co:dual}
Every dualising variety has right determined morphisms.\qed
\end{cor}

\subsection*{Minimal determinators}

Suppose a morphism is determined by two objects $C$ and $C'$. What is
then the relationship between these objects? The following proposition
gives a precise answer. For each $X\in\C$ let
$\add X$ denote the full subcategory consisting of the direct summands of
finite directs sums of copies of $X$.

\begin{prop}\label{pr:min}
Let $\C$ be a dualising variety and $\a\colon X\to Y$ a morphism in
$\C$. Then there exists in $\mod\C$ an injective envelope of the form
\[\Coker\Hom_\C(-,\a)\lto D\Hom_\C(C,-)\]
for some object $C\in\C$.  
Given an object $C'$ in $\C$, the morphism $\a$ is
right $C'$-determined if and only if $\add C\subseteq\add C'$.
\end{prop}
\begin{proof}
  The category $\mod\C$ has projective covers since the endomorphism
  ring of each object in $\C$ is semiperfect; see
  \cite[Proposition~A.1]{Kr2011}.  Applying the duality, it follows
  that each object $F$ has an injective envelope of the form $F\to
  D\Hom_\C(C,-)$ for some $C\in\C$.  Now set $F=\Coker\Hom_\C(-,\a)$.
  Then Proposition~\ref{pr:det-obj} shows that $\a$ is right
  $C$-determined.  Given an object $C'\in\C$, it follows that $\a$ is
right $C'$-determined if $\add C\subseteq\add C'$.

Now suppose that $\a$ is right $C'$-determined.  The proof of
Proposition~\ref{pr:det-exist} yields a monomorphism $F\to\coind_{C'}
I$ for some injective $\End_\C(C')$-module $I$.  Here, we use the
uniqueness of a right determined morphism; see Remark~\ref{re:uni}.
From Lemma~\ref{le:Yon} it follows that $\coind_{C'} I$ is a direct
summand of a finite direct sum of copies of $D\Hom_\C(C',-)$.  On the
other hand, the assumption on $C$ implies that $D\Hom_\C(C,-)$ is a
direct summand of $\coind_{C'} I$, since $\coind_{C'} I$ is an
injective object.  Thus $C$ is a direct summand of a finite direct sum
of copies of $C'$.
\end{proof}

\section{Triangulated categories with Serre duality}

Fix a $k$-linear triangulated category $\C$ which is Hom-finite,
essentially small, and idempotent complete. 

Recall from \cite{RvB2002} that a \emph{right Serre functor} is an
additive functor $S\colon\C\to\C$ together with a natural isomorphism
\[\eta_X\colon D\Hom_\C(X,-)\xto{\sim}\Hom_\C(-,SX)\] for all $X\in\C$,
where $D=\Hom_k(-,E(k/\mathfrak r))$. 
A right Serre functor is called a \emph{Serre functor} if it is an
equivalence.

\begin{prop}\label{pr:main}
Consider for the category $\C$ the following conditions:
\begin{enumerate}
\item The category $\C$ admits a right Serre functor  $S\colon\C\to\C$.
\item Given objects $C,Y\in\C$ and an $\End_\C(C)$-submodule
  $H\subseteq\Hom_\C(C,Y)$, there is a right $C$-determined morphism
  $\a\colon X\to Y$ with $\Im\Hom_\C(C,\a)=H$.
\item Every morphism in $\C$ is left determind by an object in $\C$.
\end{enumerate}
Then \emph{(1)} $\iff$ \emph{(2)} $\implies$ \emph{(3)}.
\end{prop}
\begin{proof}
Observe first that a  triangulated category has weak kernels and cokernels.

That (1) implies (2) follows from Proposition~\ref{pr:det-exist}, and
that  (1) implies (3) follows from Corollary~\ref{co:det}.

It remains to show that (2) implies (1). From
\cite[Proposition~I.2.3]{ RvB2002} it follows that $\C$ has a right
Serre functor if and only if there is an Auslander--Reiten triangle
ending at each indecomposable object of $\C$.  An exact triangle
$X\xto{\a} Y\xto{\b} Z\to$ is by definition an \emph{Auslander--Reiten
  triangle} ending at $Z$, if $\a$ is left almost split and $\b$ is
right almost split. This is equivalent to $\b$ being right minimal and
right almost split, by \cite[Lemma~2.6]{Kr2000}.

Now fix an indecomposable object $Z$ in $\C$. Then $\Ga=\End_\C(Z)$ is
local and there exists a right $Z$-determined morphism $\b\colon Y\to
Z$ such that $\Im\Hom_\C(Z,\b)=\rad\Ga$. We may assume that $\b$ is
right minimal by Remark~\ref{re:min}, and it follws from
Example~\ref{ex:ass} that $\b$ is right almost split.  Completing $\b$
to an exact triangle then gives an Auslander--Reiten triangle ending
at $Z$. It follows that $\C$ has a Serre functor.
\end{proof}

\begin{thm}\label{th:main}
  For a Hom-finite essentially small and idempotent complete $k$-linear
  triangulated category $\C$ the following are equivalent:
\begin{enumerate}
\item The category $\C$ admits a Serre functor  $S\colon\C\xto{\sim}\C$.
\item The category $\C$ is a dualising variety.
\item The category $\C$ has right determined morphisms.
\end{enumerate}
In this case, every morphism in $\C$ with cone $C$ is right determined
by $S^{-1}C$.
\end{thm}

\begin{proof}
  (1) $\Rightarrow$ (2): A triangulated category has weak kernels and
  cokernels. From the definition of a Serre functor, it follows that
  $D\Hom_\C(-,C)$ and $D\Hom_\C(C,-)$ are finitely presented for each
  $C\in\C$.  Thus $\C$ is a dualising variety by Lemma~\ref{le:dual}.

(2) $\Rightarrow$ (3): Apply Corollary~\ref{co:dual}.

(3) $\Rightarrow$ (1): From Proposition~\ref{pr:main} we know that
there is a right Serre functor $S\colon\C\to\C$, and it remains to
show that $S$ is an equivalence. In fact, it suffices to show that $S$
is essentially surjective on objects, since $S$ is automatically fully
faithful. Choose an object $Y$ and suppose the morphism $0\to Y$ is
right determined by some object $C$. The proof of
Proposition~\ref{pr:det-exist} yields a monomorphism
$\eta\colon\Hom_\C(-Y)\to\coind_C I$ for some injective
$\End_\C(C)$-module $I$.  Here, we use the uniqueness of a right
determined morphism; see Remark~\ref{re:uni}.  From Lemma~\ref{le:Yon}
it follows that $\coind_C I$ is a direct summand of a finite direct
sum of copies of $D\Hom_\C(C,-)\cong\Hom_\C(-,SC)$. Now one uses that
$\Hom_\C(-Y)$ is an injective object in $\mod\C$; this follows from a
direct argument  \cite[Lemma~1.6]{Kr2000a} or the fact that for every triangulated
category $\C$ the assignment $\Hom_\C(-,X)\mapsto\Hom_\C(X,-)$ induces
an equivalence
\[(\mod\C)^\op\stackrel{\sim}\lto\mod(\C^\op).\] Thus $\eta$ is a
split monomorphism, and it follows that $Y$ is a direct summand of
some object in the image of $S$. Using that $\C$ is idempotent
complete, it follows that $Y\cong SX$ for some $X\in\C$.

To complete the proof, we fix an exact triangle $X\xto{\a} Y\xto{}
SC\to$ and claim that $\a$ is right determined by $C$. This
an immediate consequence of Proposition~\ref{pr:det-obj}, since the
triangle induces an exact sequence
\[\Hom_\C(-,X)\xto{(-,\a)} \Hom_\C(-,Y)\lto D\Hom_\C(C,-).\qedhere\]
\end{proof}

\section{A generalisation}

In this section we generalise Auslander's definition of a right
determined morphism as follows; see also \cite[\S4]{Kr2000}.

\begin{defn}
  A morphism $\a\colon X\to Y$ in a category $\C$ is said to be
  \emph{right determined} by a class $\D$ of objects of $\C$ if for
  every morphism $\a'\colon X'\to Y$ the following conditions are
  equivalent:
\begin{enumerate}
\item The morphism $\a'$ factors through $\a$.
\item For every morphism $\p\colon C\to X'$ with $C\in\D$ the composite $\a'\p$
  factors through $\a$.
\end{enumerate}
\end{defn}

Let $\C$ be an essentially small additive category. We denote by
$\Mod\C$ the category of additive functors $\C^\op\to\Ab$. Given a
full additive subcategory $\D$ of $\C$, we consider the restriction
functor $\res_\D\colon\Mod\C\to\Mod\D$ and its right adjoint
$\coind_\D\colon\Mod\D\to\Mod\C$ with
\[(\coind_\D F)(X)=\Hom_\D(\res_\D\Hom_\C(-,X),F)\]
for $F\in\Mod\D$ and $X\in\C$.

Fix a morphism $\a\colon X\to Y$ in $\C$. Then a morphism $\a'\colon
X'\to Y$ factors through $\a$ if and only if
\begin{equation}\label{eq:det1}
\Im\Hom_\C(-,\a')\subseteq\Im\Hom_\C(-,\a).
\end{equation}
This condition implies
\begin{equation}\label{eq:det2}
\res_\D\Im\Hom_\C(-,\a')\subseteq\res_\D\Im\Hom_\C(-,\a).
\end{equation}
Reformulating the above definition, the morphism $\a$ is determined by
$\D$ if and only if \eqref{eq:det1} and \eqref{eq:det2} are equivalent
for all $\a'\colon X'\to Y$.

The following proposition characterises the morphisms that are right
determined by a fixed class of objects. This provides a conceptual
explanation for some of the previous results.

\begin{prop}
  Let $\C$ be an essentially small additive category and $\D$ a full
  additive subcategory. For a morphism $\a\colon X\to Y$ the following are
  equivalent:
\begin{enumerate}
\item The morphism $\a$ is right determined by $\D$.
\item For $F=\Coker\Hom_\C(-,\a)$ the canonical morphism $F\to
  \coind_\D\res_\D F$ is a monomorphism.
\item For some $I\in\Mod \D$ there is an exact sequence
\[\Hom_\C(-,X)\xto{(-,\a)}\Hom_\C(-,Y)\lto\coind_\D I.\]
\end{enumerate}
\end{prop}
\begin{proof}
  (1) $\Leftrightarrow$ (2): A morphism $\p\colon C\to Y$ in $\C$
  yields an element $\bar\p\in F(C)$, and $\bar\p=0$ if and only if
\[\Im\Hom_\C(-,\p)\subseteq\Im\Hom_\C(-,\a).\] 
Now consider the canonical morphism $\eta\colon F\to
\coind_\D\res_\D F$ and observe that $\eta_C(\bar\p)=0$ if and only if
\[\res_\D\Im\Hom_\C(-,\p)\subseteq\res_\D\Im\Hom_\C(-,\a).\] 
Thus $\a$ is right determined by $\D$ if and only if 
$\eta_C$ is a monomorphism for all $C\in\C$.

(2) $\Rightarrow$ (3): Take $I=\res_\D F$.

(3) $\Rightarrow$ (2): Every morphism $\theta\colon F\to \coind_\D I$
factors through the canonical morphism $\eta\colon F\to
\coind_\D\res_\D F$. It follows that $\eta$ is a monomorphism if $\theta$
is a monomorphism.
\end{proof}

The general definition of a morphism determined by a class of objects
suggests the following question.

\begin{question}
  Given a morphism in some category $\C$, is there a \emph{minimal} class
  $\D$ of objects of $\C$ such that the morphism is right
  $\D$-determined?
\end{question}

We have seen in Proposition~\ref{pr:min} that such minimal
determinators always exist for dualising varieties. Next we discuss
a classical problem from stable homotopy theory. It turns out that
Freyd's generating hypothesis predicts a  determinator for a
particular class of morphisms.

\section{Freyd's generating hypothesis}

We consider the stable homotopy category of spectra and the
set $\Sc= \{\Sigma^n S\mid n\in\bbZ\}$ formed by the
suspensions of the sphere spectrum $S$.

\begin{thm}
The following conditions are equivalent:
\begin{enumerate}
\item Freyd's generating hypothesis holds, that is, for every finite
  spectrum $Y$, the morphism $0\to Y$ is right $\Sc$-determined as
  morphism in the category of finite spectra \cite[\S9]{Fr1966} .
\item For every finite torsion spectrum $Y$,
 the morphism $0\to Y$ is right $\Sc$-determined as  morphism in
 the category of finite spectra.
\item For every finite torsion spectrum $Y$,
  the morphism $0\to Y$ is  right $\Sc$-determined as morphism in the
  category of all spectra.
\end{enumerate}
\end{thm}
\begin{proof}
  (1) $\Leftrightarrow$ (2): One direction is clear, and the other
  follows from the fact that the torsion spectra cogenerate the
  category of finite spectra; see \cite[Proposition~6.8]{Fr1966}.

  (2) $\Leftrightarrow$ (3): One needs to show that $0\to Y$ is a
  right determined morphism in the category of all spectra if it is
  right determined in the category of finite spectra. Observe first
  that every finite torsion spectrum $Y$ is \emph{endofinite}
  \cite[\S1]{Kr1999}. More precisely, $\Hom(F,Y)$ has finite length as
  an $\End(Y)$-module for each finite spectrum $F$. It follows from
  \cite[Theorem~1.2]{Kr1999} that for each non-zero morphism $X\to Y$
  the induced map $\Hom(F,X)\to\Hom(F,Y)$ is non-zero for some finite
  spectrum $F$. Thus there is some non-zero morphism $F\to X\to Y$.
  If (2) holds, this implies that for some $n\in\bbZ$
  there is a morphism $\Sigma^nS\to F$ such that the composite
$\Sigma^nS\to F\to X\to Y$ is non-zero. Thus the
  morphism $0\to Y$ is determined by $\Sc$.
\end{proof}

\subsection*{Acknowledgement} 
Some 20 years ago, Maurice Auslander encouraged me (then a postdoc at
Brandeis University) to read his Philadelphia notes \cite{Au1978},
commenting that they had never really been used. More recently,
postdocs at Bielefeld asked me to explain this material; I am grateful
to both of them. Special thanks goes to Greg Stevenson for helpful
discussions and comments on a preliminary version of this paper.

\end{document}